\documentclass[a4paper,10pt]{article}
\usepackage{cite}
\usepackage[top=3cm, bottom=2cm, left=3cm, right=3cm]{geometry}
\usepackage[utf8]{inputenc}
\usepackage[T1]{fontenc}
\usepackage[english]{babel}
\usepackage[super]{nth}

\usepackage{hyperref}
\usepackage{url}
\usepackage{stmaryrd}
\usepackage{amsmath}
\usepackage{amssymb}
\usepackage{mathbbol}
\usepackage{mathtools}
\usepackage{amsthm}
\usepackage{bm,accents}
\usepackage{appendix}
\usepackage{xfrac}
\usepackage{relsize}
\usepackage{faktor}
\usepackage[all]{xy}
\usepackage{thmtools, thm-restate}
\newtheorem{theorem}{Theorem}[section]

\newtheorem{lemma}[theorem]{Lemma}
\newtheorem{proposition}[theorem]{Proposition}
\newtheorem{corollary}[theorem]{Corollary}
\newtheorem{definition}[theorem]{Definition}

\newcommand{\cP}{\mathcal{P}}

\newcommand{\powerset}{\cP}
\newcommand{\inl}{{\sf inl}}
\newcommand{\inr}{{\sf inr}}
\newcommand{\Ninfty}{{\mathbb{N}_\infty}}

\newcommand{\bbN}{\mathbb N}

\newcommand{\limplies}{\Rightarrow}

\newcommand{\eqdef}{\mathrel{:=}}
\newcommand{\Succ}{{\sf S}}

\newcommand{\Ninftys}{\underline{\Succ}}
\newcommand{\op}{\text{op}}

\newcommand{\finset}{{\sf Finset}}

\newcommand{\CB}{{\mathbb{CB}}}
\newcommand{\EM}{{\mathbb{EM}}}
\newcommand{\CBBB}{{\mathbb{CBBB}}}

\newcommand{\LPO}{{\mathbb{LPO}}}

\begin{document}

\title{Cantor-Bernstein implies Excluded Middle}

\author{
C\'ecilia Pradic~\thanks{Swansea University}
\and
Chad E. Brown\thanks{Czech Technical University in Prague}
}

\maketitle

\begin{abstract}
We prove in constructive logic that the statement of the Cantor-Bernstein theorem implies excluded middle.
This establishes that the Cantor-Bernstein theorem can only be proven assuming the full power of classical logic.
The key ingredient is a theorem of Mart\'{\i}n Escard\'{o} stating that quantification over a particular subset of the Cantor space $2^\bbN$, the so-called one-point compactification of $\bbN$, preserves decidable predicates.
\end{abstract}

The \emph{Cantor-Bernstein theorem} is an elementary statement $\CB$ of set theory: for any two sets $A$ and $B$, if there are injections $f : A \hookrightarrow B$ and $g : B \hookrightarrow A$, then there exists a bijection $h : A \xrightarrow{\sim} B$.
An interesting feature of this theorem is that it may be proven in ZF without assuming the axiom of choice.
However, this proof is non-constructive in the sense that it goes through classical logic; while the construction of the bijection $h$ is rather explicit, one needs to appeal to excluded middle to show that it is indeed a bijective function.

Models of constructive set theory invalidating $\CB$ are known, such as for instance, models based on Kleene realizability\footnote{For instance, consider the subset $H \subseteq \bbN$ corresponding to the halting problem. If $\CB$ held in the effective topos, we would have a recursive bijection $\bbN \xrightarrow{\sim} \{ 2n \mathrel{|} n \notin H \} \cup \{ 2n + 1 \mathrel{|} n \in \bbN \}$; since it is in particular surjective, we would be able to build a recursive enumeration of the complement of $H$, which is absurd.}.
However, this still left open the question of whether the full power of \emph{excluded middle} $(\EM)$ is really necessary to prove the theorem or if a
weaker classical principle would be enough.
The purpose of this note is to show that, indeed, full excluded middle is required.  

\begin{restatable}{maintheorem}{maincbxm} $(\CB \Rightarrow \EM)$ \label{thm:maincbxm}
  Over intuitionistic set theory, the Cantor-Bernstein theorem implies the principle of excluded middle.
\end{restatable}

The argument is a straightforward application of a key theorem of Mart\'{\i}n Escard\'{o}~\cite{Esc13} concerning 
the one-point compactification $\Ninfty$ of $\bbN$ and decidable predicates:
there exists a function
\[ \varepsilon \quad : \quad (\Ninfty \to 2) \quad \longrightarrow \quad \Ninfty\]
selecting a counter-witness for its input when possible. Formally speaking, it means that for any decidable predicate $P : \Ninfty \to 2$, we have the following equivalence\footnote{This should not be confused with Hilbert's $\varepsilon$ which dually selects \emph{witnesses} of existential statements.}.
\[ \big(\forall p \in \Ninfty. \; P(p) = 1\big)   \qquad \Longleftrightarrow \qquad P(\varepsilon(P)) = 1 \]
This result is rather striking as it means that there exists an infinite set\footnote{By which we systematically mean Dedekind-infinite here.} for which decidable predicates are stable under quantification, \emph{provably in constructive logic}. This is to be contrasted against the case of $\bbN$, which admits no recursive selection function.

Assuming the existence of \emph{any} infinite set equipped with a selection function, $\CB \Rightarrow \EM$ may be proven by very elementary means.
The reader familiar with~\cite{Esc13} may content themselves with the proofs of Lemma~\ref{lem:surinh} and Theorem~\ref{thm:maincbxm}. Section~\ref{sec:cbxm} as a whole gives a self-contained proof of $\CB \Rightarrow \EM$, integrating the necessary technical content from~\cite{Esc13} about $\Ninfty$.
For the more casual reader, we first give a preliminary example of an elementary set-theoretic statement implying excluded middle in Section~\ref{sec:elemex}, so as to illustrate how many similar statements can
be proven because of the comprehension axiom.
We then give an already-known proof~\cite{Banaschewski-1986}
 of a weaker variant of Theorem~\ref{thm:maincbxm} in Section~\ref{sec:cbbbem} for didactic purposes, before moving on to the proof of Theorem~\ref{thm:maincbxm}.

\subparagraph*{Foundational and notational preliminaries}
We work informally in an intuitionistic set theory such as IZF where we do not assume excluded middle upfront;
we essentially require that the usual axioms of Zermelo's set theory hold~\cite{myhill1973some}, including full
separation.
Among other things, this means that our universe of sets is closed under pairing, union, powerset
(which we write $\powerset(-)$) and set comprehension $\{ x \in A \mid \varphi \}$.
For sets $A$ and $B$, function spaces $B^A$ and disjoint unions $A + B$ are built as usual.
We write $\inl:A\to A+B$ and $\inr:B\to A+B$ for the usual injections into disjoint unions.
In particular, for every $x\in A+B$ there is either an $a\in A$
such that $x=\inl(a)$ or $b\in B$ such that $x=\inr(b)$.
We let $0$ be the empty set $\emptyset$, $1 = \{0\}$ and $2=\{0,1\} \cong 1 + 1$.

In set theory, propositions or truth values can be arranged as a set $\Omega = \powerset(1)$, which
is closed under all logical connective.
For a formal proposition $p \in \Omega$, we sometimes abbreviate ``$p = 1$'' by ``$p$'' when writing
formulas. \emph{Excluded middle} $\EM$ may then be formally written as
$\forall p \in \Omega. \; p \vee \neg p$\footnote{Note that the scheme $\forall x. \; \varphi(x) \vee \neg \varphi(x)$ for arbitrary $\varphi$ can be recovered by set comprehension.}.
Note that an equivalent formulation of $\EM$ in this setting is $\Omega \cong 2$,
which is \emph{not} the case in constructive settings.
In the sequel, although they are definitionally the same in set theory, we often make a notational distinction between propositions $p \in \Omega$ and subsets $A \subseteq 1$ for readability.

Finally, and crucially\footnote{Interestingly, one may easily construct models of finitary models satisfying $\CB$ but not $\EM$: take the internal logic of the topos $\finset^{\mathcal{C}^\op}$. If $\mathcal{C}$ is not a groupoid, $\EM$ is not satisfied, while $\CB$ always hold in the internal logic. This indicates that assuming the axiom of infinity is essential here.}, we assume the axiom of infinity and call $\bbN$ the set of natural numbers.

\subparagraph*{Related works}
We do not reprove $\EM \Rightarrow \CB$ here; while it is not necessary
to read this note, they motivate a strengthening of $\CB$ mentioned in Definition~\ref{def:cbbb}.
Any introduction to set theory should have a satisfactory proof; for reference, one may look at~\cite[Theorem 3.2]{jech}.
For a historical perspective, the reader may be interested in Hinkis' monograph~\cite{hinkis}.
The question of the non-constructiveness of $\CB$ from the categorical point of view was
studied by Banaschewski and Br\"ummer in~\cite{Banaschewski-1986} and mentioned in Johnstone's Elephant~\cite[Lemma D4.1.12, p.950]{Johnstone02}.

Since the first version of this note, Mart\'{i}n Escard\'{o} carried out further investigations
regarding Cantor-Bernstein and excluded middle in the context of type theory;
the Agda formalization~\cite{EscCBAgda} includes a wealth of further results in
addition to a generalization of $\CB$ to $\infty$-groupoids\cite{EscGrpd}.

\subparagraph*{Acknowledgements}
Thanks to Robert Passman who signaled us a misleading statment in the initial version of this document,
which misguidedly mentioned CZF (full separation is needed here if we want full
$\EM$) and to Jim Kingdon for reporting mistaken type annotations.
Let us also thank (belatedly) Andrej Bauer and Henryk Michalewski for pushing us to put this out there
and Mart\'{i}n Escard\'{o} for follow-up discussions. Also many thanks to them and the
broader community for their inspiring work!

\section{An elementary set-theoretic statement implying $\EM$}
\label{sec:elemex}
Let us start with an example of a statement involving functions, which could very well be given as a exercise in a first introduction to set theory.

\begin{proposition}
  \label{prop:injsur}
  Let $A$ and $B$ be sets and $f : A \to B$ an injective function. Suppose that $A$ is non-empty.
  If excluded middle holds,
  then there exists a surjection $g : B \to A$.
\end{proposition}
\begin{proof}
Since $A$ is non-empty, one can pick a default element $d \in A$.
By excluded middle, we know for every $y\in B$ either there exists an $x\in A$ such that $f(x)=y$
or no such $x$ exists.
Hence we can define $g$ by cases as follows:
$$
g(y) := \left\{
\begin{array}{ll}
  x & {\mbox{ if }} f(x)=y \\
  d & {\mbox{ if no such }} x {\mbox{ exists}}
\end{array}
\right.
$$
Note $g$ is well defined since $f$ is injective,
and $g$ is clearly surjective.
\end{proof}

Notice that in this little proof, one needs to make a case analysis using excluded middle. As in the case of Cantor-Bernstein, one can ask if this is necessary. It is in our setting.

\begin{proposition}
  \label{prop:injsurxm}
  Suppose for all sets $A$ and $B$ there is a surjective function $g:B\to A$ whenever $A$ is nonempty and there is an injective function $f:A\to B$.
  Then excluded middle holds.
\end{proposition}
\begin{proof}
  Let $p\in\Omega$ be given
  and consider $A \eqdef \{ 0 \mathrel{|} p \}$.
  There is clearly an injection $f:A+1\to 2$ given by $f(\inl(x)) = 0$ for $x\in A$ and $f(\inr(0)) = 1$.
  Note that $A + 1$ is non-empty, as $\inr(0) \in A + 1$.
  Applying our assumption, we obtain a surjection $g : 2 \to A + 1$.
  Note that $$\forall i \in 2. \; g(i) = \inl(0) \lor g(i) = \inr(0)$$ holds because of the universal property of disjoint unions.
  Hence we have two cases.
  \begin{itemize}
  \item Suppose $g(i) = \inl(0)$ for some $i\in 2$.
    In this case $0\in A$ and so $p$ holds.
  \item Suppose $g(0) = \inr(0)$ and $g(1) = \inr(0)$.
    We prove $\neg p$ holds. To this end, assume $p$ holds.
    Hence $0\in A$. Since $g$ is surjective there must be some $i\in 2$ such that $g(i)=\inl(0)$,
    contradicting our assumption.
  \end{itemize}
\end{proof}

Here, the strategy was fairly simple: take the $A \subseteq 1$ associated to the proposition, and try to make it fit in the hypothesis of the lemma
using disjoint unions and singletons. The situation in the proof can be visualized as follows:
$$
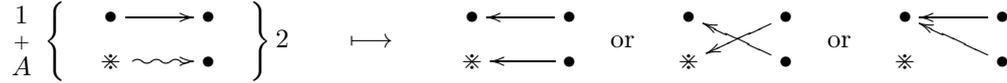
\vcenter{\xymatrix@R=7pt{ 1 \ar@{}[d]|+ \\ A}}
\left\{
\quad
\vcenter{
\xymatrix@R=5pt{
 \bullet \ar[r] & \bullet \\
\divideontimes \ar@{~>}[r] & \bullet \\
}}
\quad \right\} 2
\hspace{20pt}
\longmapsto
\hspace{20pt}
\vcenter{\xymatrix@R=5pt{
\bullet & \ar[l] \bullet \\
\divideontimes & \ar[l] \bullet \\
}}
\text{\quad or \quad}
\vcenter{\xymatrix@R=5pt{
\bullet & \ar[dl] \bullet \\
\divideontimes & \ar[ul] \bullet \\
}}
\text{\quad or \quad}
\vcenter{\xymatrix@R=5pt{
\bullet & \ar[l] \bullet \\
\divideontimes & \ar[ul] \bullet \\
}}
$$

\section{Reversing Cantor-Bernstein-Banaschewski-Br{\"{u}}mmer}
\label{sec:cbbbem}

One can try to adopt a similar strategy for proving that Cantor-Bernstein implies excluded middle.

Provided some $A \subseteq 1$, one can start building an injection $f : A \to 1$.
$$\xymatrix@R=5pt{\divideontimes \ar@{~>}[r] & \bullet}$$
However, we need to have an injection going back and we are unsure of the existence of an element in $A$. So let us consider the obvious injection $1 \to A + 1$.
$$
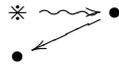
\xymatrix@R=5pt{\divideontimes \ar@{~>}[r] & \ar[dl] \bullet \\
            \bullet & }$$
Again we need a new value to be the image under $f$ of this latest element, which leads us to consider $2$.
Since we still do not have two injections, one might be tempted to iterate this process.
$$\vcenter
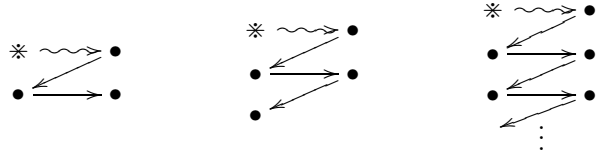
{\xymatrix@R=5pt{\divideontimes \ar@{~>}[r] & \ar[dl] \bullet \\
            \bullet \ar[r] & \bullet }}
\hspace{40pt}
\vcenter{\xymatrix@R=5pt{\divideontimes \ar@{~>}[r] & \ar[dl] \bullet \\
            \bullet \ar[r] & \ar[dl] \bullet \\
            \bullet &}}
\hspace{40pt}
\vcenter{\xymatrix@R=5pt{\divideontimes \ar@{~>}[r] & \ar[dl] \bullet \\
            \bullet \ar[r] & \ar[dl] \bullet \\
            \bullet \ar[r] & \ar[dl] \bullet \\
            \ar@{}[r]|{\vdots} & }}$$
This informal discussion suggests using $\bbN$ and the following injections.
$$
\vcenter{\xymatrix@R=30pt{ A \ar@{}[d]|+ \\ \bbN}}
\left\{ \quad \vcenter{\xymatrix@R=5pt{\divideontimes \ar@{~>}[r] & \ar[dl] \bullet \\
            \bullet \ar[r] & \ar[dl] \bullet \\
            \bullet \ar[r] & \ar[dl] \bullet \\
            \ar@{}[r]|{\vdots} & \\
            \ar@{}[r]|{\vdots} & \\
}} \quad \right\} \bbN
\hspace{40pt}
\begin{array}{lrclcllcl}
f : & \bbN &\longrightarrow& A + \bbN \\ 
    & n &\mapsto& \inr(n) \\ 
    &   &        &        \\ 

g : &A + \bbN &\longrightarrow& \bbN \\
&\inl(0) &\mapsto& 0 \\
&\inr(n) &\mapsto& n+1

\end{array}
$$

$\CB$ then provides a bijection $h : \bbN \to A + \bbN$. In fact, in elementary
proofs of the theorem this bijection can be seen as a perfect matching of the above
graph. Note however that the usual statement $\CB$ conceals this relationship between $f, g$
and $h$. Banaschewski and Br\"ummer~\cite{Banaschewski-1986} studied the corresponding
strengthened version of $\CB$, which we dub $\CBBB$, in a categorical setting and proved
that it implied excluded middle.

\begin{definition}\label{def:cbbb}
  We say {\emph{${\CBBB}$ holds}} if the following statement holds:
  given sets $A$ and $B$ and injections $f:A\to B$ and $g:B\to A$,
  there is a bijection $h:A\to B$
  such that for all $x\in A$ and $y\in B$,
  $f(x) = y$ or $x = g(y)$
  whenever $h(x) = y$.
\end{definition}

Let us remark that it is obvious that $\CBBB \Rightarrow \CB$.
Let us also stress that $\EM \Rightarrow \CBBB$ can be easily obtained
 by adapting elementary proofs of $\EM \Rightarrow \CB$.

\begin{theorem}[Proposition 4.1 in~\cite{Banaschewski-1986}]\label{thm:cbbbxm}
  If ${\CBBB}$ holds, then excluded middle holds.
\end{theorem}

\begin{proof}
  Assume ${\CBBB}$ holds.
  Let a proposition $p\in \Omega$ be given, seen as a subset $A = \{ 0 \mathrel{|} p \} \subseteq 1$.
  Take $f:\bbN\to A+\bbN$ and $g:A+\bbN\to \bbN$ to be the injections described above.

  By ${\CBBB}$ there is a bijection $h:\bbN \to A+\bbN$ such that
  $f(x) = y$ or $g(y) = x$
  whenever $h(x) = y$ for $x\in \bbN$ and $y\in A+\bbN$.
  We know either $h(0) = \inl(0)$ or $h(0) = \inr(n)$ for some $n\in\bbN$.
  \begin{itemize}
  \item
  If $h(0) = \inl(0)$, then $0\in A$ and so $p$ holds.
  \item
  Suppose $h(0) = \inr(n)$ for some $n\in\bbN$.
  We will prove $\neg p$ holds. To this end, assume $p$ holds, so that $0\in A$.
  Since $h$ is surjective, there is some $m\in \bbN$ such that $h(m) = \inl(0)$.
  Either $f(m) = \inl(0)$ or $g(\inl(0)) = m$.
  The first case is impossible since $f(m) = \inr(m)$ by the definition of $f$ and $\inl(0)\not=\inr(m)$.
  Therefore, $g(\inl(0)) = m$ and, by definition of $g$, we must have $m=0$.
  This is also impossible since it implies $\inl(0)=h(m)=h(0)=\inr(n)$.
  \end{itemize}
\end{proof}

\section{Reversing Cantor-Bernstein}
\label{sec:cbxm}

Let us pause a moment and consider why we failed to prove the analogue of Proposition~\ref{prop:injsurxm}.
In that proof of that proposition lemma, we did not use any information about the surjection $g$.
Instead, we resorted to exhaustively enumerating the set $2$ to check whether we had some $x \in 2$
such that $g(x) = \inr(0)$, which is a decidable property.
This feature of $2$ of being searchable may be formalized using the
notion of \emph{omniscience}.

\begin{definition}[Omniscient sets]\label{def:omn}
We say a set $O$ is \emph{omniscient} if for every $p \in 2^O$ if either there exists $x \in O$ such that $p(x) = 0$, or $p$ is constantly equal to $1$. That is,
$$\forall p \in {2^O}. (\exists x \in O. p(x) = 0) \vee (\forall x \in O. p(x) = 1)$$
\end{definition}

In classical logic, all sets are clearly omniscient, but this is not necessarily true in constructive logics.
However, all finite sets, and in particular $2$, are omniscient.
This concept allows us to isolate the actual core of the proof of Proposition~\ref{prop:injsurxm}.

\begin{lemma}
\label{lem:surinh}
Suppose that we have an omniscient set $O$ and some sets $A$ and $B$. If there exists a surjection $f : O \to A + B$, then either $A$ is inhabited or it is empty.
\end{lemma}
\begin{proof}
  Let $f:O\to A+B$ be given and define $p\in 2^O$ by
  $$
  p(x) = \left\{
  \begin{array}{lr}
    0 & {\mbox{ if }} \exists a\in A. \; f(x) = \inl(a) \\
    1 & {\mbox{ if }} \exists b\in B. \; f(x) = \inr(b)
  \end{array}
  \right.
  $$
  Since $O$ is omniscient either $\exists x. \; P(x) = 0$ or $\forall x\in O. \; P(x) = 1$.
  If $\exists x\in O. \; P(x) = 0$, then $A$ is clearly inhabited.
  Suppose $P(x) = 1$ for every $x\in O$.
  We show that $A$ is empty. Suppose that we have $a\in A$ toward a contradiction. Since $f$ is surjective, there is some $x\in O$ such that $f(x) =\inl(a)$ and thus $P(x) = 0$.
This contradicts $P(x) = 1$.
\end{proof}

From ${\CB}$ (instead of the stronger ${\CBBB}$)
we could use the injections from the proof of Theorem~\ref{thm:cbbbxm}
to obtain a surjection $\bbN \to A + \bbN$.
If $\bbN$ were omniscient, then we could use this surjection with Lemma~\ref{lem:surinh}
to $\EM$.
However, omniscience of $\bbN$
correspond to the axiom of \emph{limited principle of omniscience} ($\LPO$)
a well-known constructive taboo~\cite{Bishop1967},
which can be thought of as a (strictly weaker) version of $\EM$\footnote{
Remark that, at this point, we have $\LPO \wedge \CB ~ \Rightarrow ~ \EM$ over constructive set theory. This observation is however not necessary to carry out the subsequent argument.}.
This means that deriving $\EM$ from the existence of bijections $\bbN \cong 1 + \bbN$ by way of Lemma~\ref{lem:surinh} is unreasonable.
Luckily for us, Escard\'{o} proved that there exists an infinite subset of the Cantor space, $\Ninfty$, which is omniscient~\cite{Esc13} and can be used to prove
$\CB \Rightarrow \EM$.

In order to keep the argument self-contained, we reproduce his argument below before deriving the main result.

\begin{definition}
\label{def:Ninfty}
  We define 
  $\Ninfty$ to be the set of non-increasing
  sequences in $2^\bbN$, i.e.\footnote{For more categorically-inclined people, $\Ninfty$ is the final coalgebra for the functor $X \mapsto 1 + X$. This would justify calling $\Ninfty$ \emph{the} set of conatural numbers. On the other hand, the induced topology from $2^\bbN$ in Definition~\ref{def:Ninfty} also justifies calling $\Ninfty$ the \emph{one-point compactification of $\bbN$} (seen as a discrete space).},
  $$\Ninfty = \big\{p\in 2^\bbN \mathrel{\big|} \forall n\in\bbN. ~ \big( p(n) = 1 \; \Rightarrow \;  \forall m\in\bbN. ~ (m < n \; \limplies \; p(m) = 1) \big) \big\}.$$
  Let $\omega\in \Ninfty$ denote the constant function $n \mapsto 1$.
  Define an injection taking $n\in\bbN$ to $\underline{n}\in\Ninfty$ by
  $$
  \underline{n}(m) =
  \left\{
  \begin{array}{lr}
    1 & {\mbox{ if }} m < n \\
    0 & {\mbox{ otherwise.}}
  \end{array}
  \right.
  $$
  Finally, define $\Ninftys:\Ninfty\to\Ninfty$
  by setting 
  $\Ninftys(p)(0) = 1$
  and
  $\Ninftys(p)(n+1) = p(n)$.
\end{definition}

The set $\Ninfty$ is infinite 
as witnessed by $\underline{0}$ and $\Ninftys$.
\begin{lemma}\label{lem:Ninftys}
  The function $\Ninftys:\Ninfty\to\Ninfty$ is injective
  and $\Ninftys(p)\not= \underline{0}$ for all $p\in\Ninfty$.
\end{lemma}
\begin{proof}
  Suppose $\Ninftys(p) = \Ninftys(q)$.
  Since $p(n) = \Ninftys(p)(n+1) = \Ninftys(q)(n+1) = q(n)$
  for every $n\in\bbN$
  we know $p = q$, as desired.
  The fact that $\Ninftys(p)\not= \underline{0}$ for all $x\in\Ninfty$
  follows from $\Ninftys(p)(0) = 1 \not= 0 = \underline{0}(0)$.
\end{proof}

Classically every element of $\Ninfty$ is either $\omega$ or of the form $\underline{n}$.
The corresponding disjunction is equivalent to $\LPO$, and so is unprovable constructively.
However, for decidable predicates, it is sufficient to show that they hold over all elements
$\underline{n}$ and $\omega$ to show they hold everywhere\footnote{This constitutes a particular case of Lemma 3.4 in~\cite{Esc13}.}.
\begin{lemma}\label{lem:Ninftydensity}
  Let $Q\in 2^\Ninfty$ be given.
  If $Q(\omega) = 1$ and $\forall n\in\bbN. ~ Q(\underline{n}) = 1$, then
  $\forall p\in \Ninfty. ~ Q(p) = 1$.
\end{lemma}
\begin{proof}
  Let $Q\in 2^\Ninfty$ such that
  $Q(\omega) = 1$ and $\forall n\in\bbN. ~ Q(\underline{n}) = 1$.
  Let $p\in \Ninfty$ be given.
  To prove $Q(p) = 1$, it is enough to prove $Q(p)\not=0$.
  Assume $Q(p) = 0$.
  Under this assumption we can prove $\forall n\in\bbN. ~p(n) = 1$
  by strong induction.
  Assume $\forall k\in\bbN. ~ (k < n \limplies p(k) = 1)$ and
  $p(n)=0$. This is enough information to infer $p=\underline{n}$,
  contradicting $Q(p) = 0$ and $Q(\underline{n}) = 1$.
  To end the proof we note that $p$ must be $\omega$
  (since $\forall n\in\bbN.~ p(n) = 1$), contradicting $Q(p) = 0$ and $Q(\omega) = 1$.
\end{proof}

While the desired selection function $\varepsilon : 2^{\Ninfty} \to\Ninfty$
is rather easy to define effectively, Lemma~\ref{lem:Ninftydensity} is critical in allowing
to prove constructively that it is indeed a selection function.

\begin{theorem}[{\hspace{-0.05em}\cite[Theorem 3.15]{Esc13}}]\label{thm:Ninftysel}
  There is a function $\varepsilon : 2^{\Ninfty} \to\Ninfty$
  such that for every $Q\in 2^{\Ninfty}$,
  if $Q(\varepsilon(Q)) = 1$, then
  $\forall p\in \Ninfty. ~ Q(p) = 1$.
\end{theorem}
\begin{proof}
  For $Q\in 2^{\Ninfty}$, take $\varepsilon(Q)\in 2^{\bbN}$ to be
  $$\varepsilon(Q)(n) =
  \left\{
    \begin{array}{lr}
      1 & {\mbox{ if }} Q(\underline{k}) {\mbox{ for each }} k \le n \\
      0 & {\mbox{ otherwise.}}
      \end{array}
    \right.
    $$
    This is well-defined by recursion over $n$.
    It is easy to check that $\varepsilon(Q)\in\Ninfty$ as well.

    Assume $Q(\varepsilon(Q))=1$ and let $p\in \Ninfty$ be given.
    If $\forall k < n. \; Q(\underline{k})=1$ and $Q(\underline{n})=0$,
    then $\varepsilon(Q) = \underline{n}$ and so
    $Q(\underline{n}) = Q(\varepsilon(Q)) = 1$, contradicting $Q(\underline{n})=0$.
    Consequently, an induction proves that $Q(\underline{n})=1$ for every $n \in \bbN$.
    Thus $\varepsilon(Q) = \omega$
    and so $Q(\omega) = Q(\varepsilon(Q)) = 1$ holds as well.
    Hence $Q(p)=1$ for every $p\in \Ninfty$
    by Lemma~\ref{lem:Ninftydensity}.
\end{proof}

\begin{corollary}[{\hspace{-0.05em}\cite[Corollary 3.6]{Esc13}}]\label{lem:Ninftyomn}
  The set $\Ninfty$ is omniscient.
\end{corollary}
\begin{proof}
  Let $Q\in 2^{\Ninfty}$ be given.
  If $Q(\varepsilon(Q))=0$, then $\exists x\in \Ninfty.Q(x)=0$.
  If $Q(\varepsilon(Q))=1$, then
  $\forall x\in\Ninfty. \; Q(x)=1$
  by Theorem~\ref{thm:Ninftysel}.
\end{proof}

This completes the part of the construction we
obtained following Escard{\'{o}}~\cite{Esc13}.
We can now easily put it together with
Lemma~\ref{lem:surinh} to conclude.

\begin{proof}[Proof of Theorem~\ref{thm:maincbxm}]
  Assume ${\CB}$ holds.
  We know $\Ninfty$ is omniscient by Lemma~\ref{lem:Ninftyomn}.
  We know $\Ninftys$ is injective and
  $\Ninftys(x)\not=\underline{0}$ for every $x\in\Ninfty$
  by Lemma~\ref{lem:Ninftys}.
  Let $p\in\Omega$ be a proposition and take $A=\{0 \mathrel{|} p\}\subseteq 1$.
  Analogously to Section~\ref{sec:cbbbem}, we consider the following functions:
  $$
\begin{array}{lcclcccl}
f:&\Ninfty &\longrightarrow& A + \Ninfty&
\phantom{aaaaaaa} g:
&A + \Ninfty &\longrightarrow& \Ninfty \\
& x &\mapsto& \inr(x)
& &
\inl(0) &\mapsto& \underline{0}
\\
& & &
& &
\inr(x) &\mapsto& \Ninftys(x)
\end{array}
$$
  Both $f$ and $g$ are clearly injective, so we can
  apply ${\CB}$ to obtain a bijection $h:\Ninfty\to A+\Ninfty$.
  Lemma~\ref{lem:surinh} now implies that either $A$ is inhabited (so $p$ holds) or $A$ is empty (so $\neg p$ holds).
\end{proof}  

\bibliography{refs}

\begin{thebibliography}{Myh73}

\bibitem[Ban86]{Banaschewski-1986}
G.~C.L. Banaschewski, B.;~Br{\"{u}}mmer.
\newblock Thoughts on the {C}antor-{B}ernstein {T}heorem.
\newblock {\em Quaestiones Mathematicae}, 9, 01 1986.
\newblock \href {https://doi.org/10.1080/16073606.1986.9632106}
  {\path{doi:10.1080/16073606.1986.9632106}}.

\bibitem[Bis67]{Bishop1967}
E.~Bishop.
\newblock {\em Foundations of constructive analysis}.
\newblock McGraw-Hill, 1967.

\bibitem[Esc13]{Esc13}
M.~Escard{\'{o}}.
\newblock Infinite sets that satisfy the principle of omniscience in any
  variety of constructive mathematics.
\newblock {\em J. Symb. Log.}, 78(3):764--784, 2013.
\newblock URL: \url{http://dx.doi.org/10.2178/jsl.7803040}.

\bibitem[Esc20]{EscCBAgda}
M.~Escard{\'{o}}.
\newblock Agda proof of the {C}antor-{S}chröder-{B}ernstein {T}heorem for
  {$\infty$}-groupoids, 2020.
\newblock Available at
  \url{https://www.cs.bham.ac.uk/~mhe/agda-new/CantorSchroederBernstein.html}
  with source code at \url{https://github.com/martinescardo/TypeTopology/}
  (checked on 07/12/21).

\bibitem[Esc21]{EscGrpd}
M.~Escard{\'o}.
\newblock The {C}antor--{S}chr{\"o}der--{B}ernstein {T}heorem for
  {$\infty$}-groupoids.
\newblock {\em Journal of Homotopy and Related Structures}, 16(3):363--366,
  2021.
\newblock \href {https://doi.org/10.1007/s40062-021-00284-6}
  {\path{doi:10.1007/s40062-021-00284-6}}.

\bibitem[Hin13]{hinkis}
A.~Hinkis.
\newblock {\em Proofs of the Cantor-Bernstein Theorem: A Mathematical
  Excursion}.
\newblock Science Networks. Historical Studies, Vol. 45. Birkhäuser (Springer
  Basel), 2013.

\bibitem[Jec13]{jech}
T.~Jech.
\newblock {\em Set theory}.
\newblock Springer Science \& Business Media, 2013.

\bibitem[Joh02]{Johnstone02}
P.~T. Johnstone.
\newblock {\em Sketches of an Elephant: A Topos Theory Compendium: 2 Volume
  Set}.
\newblock Oxford University Press UK, 2002.

\bibitem[Myh73]{myhill1973some}
J.~Myhill.
\newblock Some properties of intuitionistic {Z}ermelo-{F}rankel set theory.
\newblock In {\em Cambridge Summer School in Mathematical Logic}, pages
  206--231. Springer, 1973.
\newblock \href {https://doi.org/10.1007/BFb0066775}
  {\path{doi:10.1007/BFb0066775}}.

\end{thebibliography}
\bibliographystyle{alphaurl}

\end{document}